\documentclass[a4paper,english,fontsize=10pt,parskip=half,abstracton]{scrartcl}
\usepackage{babel}
\usepackage[utf8]{inputenc}
\usepackage[T1]{fontenc}
\usepackage[a4paper,left=20mm,right=20mm,top=30mm,bottom=30mm]{geometry}
\usepackage{amsmath}
\usepackage{amsthm}
\usepackage{amssymb}
\usepackage{enumerate}
\usepackage{multirow}
\usepackage[bookmarks=false,
            pdftitle={Broué's isotypy conjecture for the sporadic groups and their covers and automorphism groups},
            pdfauthor={Benjamin Sambale},
            pdfkeywords={Broue conjecture, isotypies, sporadic groups},
            pdfstartview={FitH}]{hyperref}

\newtheorem{Theorem}{Theorem} %[section]
\newtheorem{Lemma}[Theorem]{Lemma}

\numberwithin{equation}{section}

\allowdisplaybreaks[1]

\renewcommand{\phi}{\varphi}
\newcommand{\C}{\operatorname{C}}
\newcommand{\N}{\operatorname{N}}
\newcommand{\Z}{\operatorname{Z}}

\newcommand{\Aut}{\operatorname{Aut}}

\newcommand{\Out}{\operatorname{Out}}

\newcommand{\GL}{\operatorname{GL}}
\newcommand{\SL}{\operatorname{SL}}

\newcommand{\Irr}{\operatorname{Irr}}
\newcommand{\IBr}{\operatorname{IBr}}

\newcommand{\Syl}{\operatorname{Syl}}

\newcommand{\Ker}{\operatorname{Ker}}

\newcommand{\CF}{\operatorname{CF}}

\mathchardef\ordinarycolon\mathcode`\:  %defines a nice ":=" 
 \mathcode`\:=\string"8000
 \begingroup \catcode`\:=\active
   \gdef:{\mathrel{\mathop\ordinarycolon}}
 \endgroup

\title{Broué's isotypy conjecture for the sporadic groups and their covers and automorphism groups}
\author{Benjamin Sambale\footnote{Institut für Mathematik, Friedrich-Schiller-Universität, 07743 Jena, Germany, 
\href{mailto:benjamin.sambale@uni-jena.de}{benjamin.sambale@uni-jena.de}}}
\date{\today}

\begin{document}
\frenchspacing
\maketitle
\begin{abstract}\noindent
Let $B$ be a $p$-block of a finite group $G$ with abelian defect group $D$ such that $S\unlhd G$, $S'=S$, $G/\Z(S)\le\Aut(S)$ and $S/\Z(S)$ is a sporadic simple group. We show that $B$ is isotypic to its Brauer correspondent in $\N_G(D)$ in the sense of Broué. This has been done by [Rouquier, 1994] for principal blocks and it remains to deal with the non-principal blocks. 
\end{abstract}

%\tableofcontents

\section{Introduction}
Let $B$ be a $p$-block of a finite group $G$ with abelian defect group $D$. Broué's Abelian Defect Group Conjecture asserts that $B$ is derived equivalent to its Brauer correspondent in $\N_G(D)$. This has been verified in only very few special situations. For example, the conjecture holds if $D$ is cyclic or $|D|\le 8$ (see \cite{Linckelmanncyclic,Rickardcyclic,Rouquiercyclic,LinckelmannC2C2,KKL,EatonE8}). In case $p=2$ or $|D|=9$, the conjecture is known to hold for principal blocks (see \cite{FongHarris,CravenRouquier,KK}). Moreover, Broué's Conjecture holds for $p$-solvable groups $G$ and some blocks of sporadic groups $G$ (see \cite{HarrisLinckelmann,KoshitaniNonPrin,KMN2,KMN3,KMHN,MuellerSchaps,KKW,KKW2,HollowayJ2}). 

The existence of a derived equivalence implies that there is a perfect isometry between $B$ and its Brauer correspondent. Since perfect isometries only concern the characters of $B$, it is usually much easier to show directly that there exists such a perfect isometry. In fact, one can often find an isotypy, i.\,e. a perfect isometry with additional local compatibility relations.
This indicates the presence of a splendid derived equivalence which is a strong form of a derived equivalence.
Rouquier~\cite{RouquierPerfect} showed that the isotypy version of Broué's Conjecture holds for all principal blocks of sporadic groups and their automorphism groups. The aim of this paper is to extend this result to all blocks of sporadic groups.

\begin{Theorem}\label{main}
Let $B$ be a block of a finite group $G$ with abelian defect group $D$ such that $S\unlhd G$, $S'=S$, $G/\Z(S)\le\Aut(S)$ and $S/\Z(S)$ is a sporadic simple group. Then $B$ is isotypic to its Brauer correspondent in $\N_G(D)$.
\end{Theorem}

Raphaël Rouquier informed the author that this theorem has not been proven so far.

We will make use of the following notation. 
We consider a $p$-block $B$ of a finite group $G$ with respect to a $p$-modular system $(K,\mathcal{O},F)$ where $\mathcal{O}$ is a complete discrete valuation ring with quotient field $K$ of characteristic $0$ and field of fractions $F$ of characteristic $p$. As usual, we assume that $K$ is a splitting field for $G$ and $F$ is algebraically closed. %copy
The set of irreducible (Brauer) characters of $B$ is denoted by $\Irr(B)$ (resp. $\IBr(B)$). Moreover, $k(B):=\lvert\Irr(B)\rvert$ and $l(B):=\lvert\IBr(B)\rvert$.
Let $D\le G$ be a defect group of $B$, and let $b_D$ be a Brauer correspondent of $B$ in $\C_G(D)$. Then $\N_G(D,b_D)$ is the inertial group of $b_D$ in $\N_G(D)$ and $I(B):=\N_G(D,b_D)/D\C_G(D)$ is called the \emph{inertial quotient}. A ($B$-)\emph{subsection} is a pair $(u,b_u)$ such that $u\in D$ and $b_u$ is a Brauer correspondent of $B$ in $\C_G(u)$. The corresponding generalized decomposition matrix is denoted by $Q_u\in\mathbb{C}^{k(B)\times l(b_u)}$. In particular, the ordinary decomposition matrix of $B$ is given by $Q_1$. The Cartan matrix $C_u$ of $b_u$ coincides with $Q_u^\text{T}\overline{Q_u}$ where $Q_u^\text{T}$ denotes the transpose of $Q_u$ and $\overline{Q_u}$ is the complex conjugate of $Q_u$. 
A \emph{basic set} for $b_u$ is a basis for the $\mathbb{Z}$-module of generalized characters on the $p$-regular elements of $\C_G(u)$ spanned by $\IBr(b_u)$. Under a change of basic sets the matrix $Q_u$ transforms into $Q_uS$ for some $S\in\GL(l(b_u),\mathbb{Z})$. Accordingly, the Cartan matrix becomes $S^\text{T}C_uS$. In particular, the elementary divisors of $C_u$ do not depend on the chosen basic set. Brauer also introduced the \emph{contribution matrix} \[M_u:=(m_{\chi\psi}^u)_{\chi,\psi\in\Irr(B)}=\overline{Q_u}C_u^{-1}Q_u^\text{T}\in\mathbb{C}^{k(B)\times k(B)}\] 
with respect to $(u,b_u)$. It is easy to see that $M_u$ does not depend on the basic set, but on the order of $\Irr(B)$. 

We denote the commutator subgroup of $G$ by $G'$. Moreover, an abstract cyclic group of order $n$ is denoted by $Z_n$.

\section{The general method}
Let $G$, $B$ and $D$ as in Theorem~\ref{main}. Let $b$ be the Brauer correspondent of $B$ in $\N_G(D)$.
The idea of the proof of Theorem~\ref{main} is to show first that the matrix $Q_1$ is essentially the same for $B$ and $b$. After that, we show that the matrices $Q_u$ are essentially uniquely determined by $Q_1$ (see below for details). Then the existence of a perfect isometry follows from a result by Horimoto-Watanabe~\cite{WatanabePerfIso}. In order to show that the perfect isometry is also an isotypy, we use the Brauer-Dade theory of blocks with defect group of order $p$. Altogether we will not use much information about $G$ beyond its character table coming from the ATLAS~\cite{Atlas}. Instead, we are concerned mostly with combinatorial computations with GAP~\cite{GAP47}.

\subsection{Step 1: Collect local data}
As mentioned in the introduction, we may always assume that $D$ is non-cyclic.
Now suppose that $B$ is a so-called non-faithful block, i.\,e. 
\[\Ker(B):=\bigcap_{\chi\in\Irr(B)}{\Ker(\chi)}\ne 1\]
(then necessarily $\Z(G)\ne1$). 
Since $\Ker(B)$ is a $p'$-group, $B$ dominates a unique block $\overline{B}$ of $G/\Ker(B)$ with defect group $D\Ker(B)/\Ker(B)\cong D$ and the “same” irreducible characters (see \cite[Theorem~5.8.8]{Nagao}). 
Let $\overline{b}$ be a Brauer correspondent of $\overline{B}$ in $\N_{G/\Ker(B)}(D\Ker(B)/\Ker(B))=\N_G(D)/\Ker(B)$. Then it is easy to show that $b$ dominates $\overline{b}$. Hence, also $b$ and $\overline{b}$ have the “same” irreducible characters and the same defect group. 
Moreover, it is clear that $I(B)\cong I(\overline{B})$. Since our method only relies on the irreducible characters and on the inertial quotient, it will become obvious that we only need to construct an isotypy between $\overline{B}$ and $\overline{b}$.
Therefore, we may restrict ourselves to the faithful blocks of $G$.

Now we need to discuss the case $G'<G$. Here the ATLAS~\cite{Atlas} shows that $|G/G'|=2$. It is well-known that $B$ covers a block $\widetilde{B}$ of $G'$ with defect group $D\cap G'$. Moreover, it follows from \cite[Lemma~5.5.7]{Nagao} that $\widetilde{B}$ is faithful. 
If $\widetilde{B}$ is not stable in $G$, then $B$ and $\widetilde{B}$ are Morita equivalent by the Fong-Reynolds Theorem (even the generalized decomposition numbers are equal by \cite{WatanabeFong}). 
By the transitivity of the Brauer correspondence, the same is true for $b$.
If on the other hand all irreducible characters in $\widetilde{B}$ are stable, then $B$ and $\widetilde{B}$ are even isomorphic as algebras (see \cite[Théorème~0.1]{Broueiso}). We will make use of these facts later. 
In case $p>2$ we have $D\cap G'=D$ and $\widetilde{B}$ occurs on Noeske's list~\cite{Noeske}.
We will add these blocks $B$ to the list (this will be only relevant for $p=5$, see Table~\ref{tab} below). Now let $p=2$. If $D\cap G'$ is cyclic, then $\widetilde{B}$ is nilpotent. By the Külshammer-Puig Theorem~\cite{exnilblocks}, also $B$ is nilpotent and we are done (see \cite{BrouePuig}). Hence, we may assume that $D\cap G'$ is non-cyclic. Then, $\widetilde{B}$ occurs on Noeske's list or $\widetilde{B}$ is principal. The latter case is impossible, since $J_1$ is the only sporadic group with an abelian Sylow $2$-subgroup, but $\Out(J_1)=1$. Moreover, $\Out(Co_3)=1$ and it follows that $|D|\le 8$. Hence, we are done by \cite{KKL}.

Now with the extension of Noeske's list in hand, we can collect local data. 
By \cite{RouquierPerfect}, we may assume that $B$ is not the principal block (observe that the principal block of a proper covering group is always non-faithful). 
It follows that we may assume that $p\in\{2,3,5,7\}$. By \cite{KKL}, we may assume further that $p\in\{3,5,7\}$. 
Then it turns out that $|D|=p^2$ (we are lucky here that the principal $3$-block of $O'N$ is covered only by the principal block of $O'N.2$).
The numbers $k(B)$ and $l(B)$ for the groups not listed in \cite{Noeske} can be determined easily in GAP.
By an old result by Brauer, all characters in $\Irr(B)$ have height $0$ (see \cite[Theorem~IV.4.18]{Feit}). 
In a series of papers, An and his coauthors proved that Alperin's Weight Conjecture holds for $B$ (see \cite{AnWilson} for the latest of these papers). Strictly speaking, these papers do not explicitly address the proper automorphism groups, but since we only encounter $5$-blocks with maximal defect for these groups (see below), Alperin's Conjecture can be checked with GAP easily. 
Since $D$ is abelian, this is equivalent to $l(B)=l(b)$. Since $k(B)-l(B)$ is locally determined (see below), we also have $k(B)=k(b)$.

\subsection{Step~2: Compute the decomposition and Cartan matrices of $B$ up to basic sets}
The character table of $G$ is given in the ATLAS~\cite{Atlas}. Let $A$ be the submatrix of the character table of $G$ consisting of the rows corresponding to characters in $\Irr(B)$ and the columns of $p$-singular elements. Since the character values lie in a suitable cyclotomic field, the corresponding Galois group acts on the columns of $A$. Since the entries of $A$ are algebraic integers, we can replace $A$ by an integral matrix $\widetilde{A}$ of the same shape such that $xA=0$ if and only if $x\widetilde{A}=0$ for every $x\in\mathbb{C}^{k(B)}$. After transforming $\widetilde{A}$ into its Smith normal form, we easily choose a basis for the $\mathbb{Z}$-module $\{x\in\mathbb{Z}^{k(B)}:x\widetilde{A}=0\}$. We write the basis vectors as columns of a matrix $\widetilde{Q}_1$. Then $\widetilde{Q}_1$ spans the space of generalized projective characters of $B$. In particular, there exists a matrix $S\in\GL(l(B),\mathbb{Z})$ such that $Q_1=\widetilde{Q}_1S$, i.\,e. $Q_1$ and $\widetilde{Q}_1$ coincide up to basic sets. In this way we also obtain $C_1$ up to basic sets and its elementary divisors.

\subsection{Step~3: Determine $I(B)$}
A major problem is the determination of $I(B)$ and its action on $D$. We may regard $I(B)$ as a $p'$-subgroup of $\GL(2,p)$. 
Let $\mathcal{R}$ be a set of representatives for the $I(B)$-conjugacy classes of $D\setminus\{1\}$. Then it is well-known that 
\[k(B)-l(B)=\sum_{u\in\mathcal{R}}{l(b_u)}.\]
Since we already know $k(B)-l(B)$, this gives an upper bound for $\lvert\mathcal{R}\rvert$. We have $I(b_u)\cong\C_{I(B)}(u)$. Moreover, $b_u$ dominates a block $\overline{b_u}$ of $\C_G(u)/\langle u\rangle$ with cyclic defect group $D/\langle u\rangle$ of order $p$ and $I(\overline{b_u})\cong I(b_u)$. By Dade's theory of blocks with cyclic defect groups (this special case was actually known to Brauer), we obtain $l(b_u)=l(\overline{b_u})=\lvert I(\overline{b_u})\rvert=\lvert I(b_u)\rvert=\lvert\C_{I(B)}(u)\rvert$. Hence, we have
\begin{align*}
k(B)-l(B)&=\sum_{u\in\mathcal{R}}{\lvert\C_{I(B)}(u)\rvert},\\
p^2-1&=\sum_{u\in\mathcal{R}}{\lvert I(B):\C_{I(B)}(u)\rvert}.
\end{align*}
Let $n_p$ be the multiplicity of $p$ as an elementary divisor of the Cartan matrix $C_1$. Then the theory of lower defect groups (see \cite[Section~1.8]{habil}) shows that
\[n_p\le \sum_{u\in\mathcal{R}}{\lvert\C_{I(B)}(u)\rvert}-\lvert\mathcal{R}\rvert=k(B)-l(B)-\lvert\mathcal{R}\rvert.\]
These equations give strong restrictions on the possible $p'$-subgroups $I(B)\le\GL(2,p)$. 
Moreover, An and Dietrich~\cite{AnDietrichSpor} determined all radical subgroups $R$ of sporadic groups and their normalizers. Since $D$ is radical, $I(B)$ occurs as a subgroup of some $\N_G(R)/\C_G(R)$. In case $p=3$ we can also compare $k(B)$ and $l(B)$ with the results in Kiyota~\cite{Kiyota} and Watanabe~\cite{WatanabeSD16}. Finally, some (but not all) inertial quotients are listed in \cite{Schaps}.

\subsection{Step~4: Determine the Morita equivalence class of $b$}
By a result of Külshammer~\cite{Kuelshammer}, $b$ is Morita equivalent to a twisted group algebra of the form $\mathcal{O}_\gamma[D\rtimes I(B)]$. After we have computed the Schur multiplier of $I(B)$, it turns out that there are only very few choices for the $2$-cocycle $\gamma$. If $\gamma$ is trivial, then $b$ is Morita equivalent to the group algebra of $D\rtimes I(B)$. In the non-trivial case, one can replace the inconvenient twisted group algebra by a group algebra of type $D\rtimes \widehat{I(B)}$ where $\widehat{I(B)}$ is a suitable covering group of $I(B)$ (see e.\,g. \cite[Proposition~1.20]{habil}). Then $b$ is Morita equivalent to a non-principal block of this group. Doing so, we can check if $k(b)$ and $l(b)$ coincide with the values already determined in Step~1. As a rule of thumb, in the twisted case, $l(b)$ is usually smaller (however, this is not generally true). 
In this way we identify $\gamma$ and the Morita equivalence class of $b$ uniquely. 
In our examples discussed below, we will see that $\gamma$ is always trivial.
However, we remark that a Morita equivalence does not necessarily induce an isotypy. 
Nevertheless, the ordinary decomposition matrix and the Cartan matrix of $b$ can be obtained easily (see \cite{BroueEqui}). 
From this we will find that the rows of the ordinary decomposition matrix of $b$ are almost always pairwise distinct (the exception is $2.Suz.2$ for which we need a special treatment). It follows that the characters in $\Irr(b)$ are all $p$-rational.

\subsection{Step~5: Determine the generalized decomposition matrix of $B$}
Using the character table of $G$, one can decide if $\Irr(B)$ contains $p$-conjugate characters. In our cases handled below, it turns out that almost always all irreducible characters of $B$ are $p$-rational (again the exception is $2.Suz.2$). It follows that the matrices $Q_u$ are integral for every $u\in\mathcal{R}$. 
The direct computation of $Q_u$ is usually not practical, since the class fusion from $\C_G(u)$ to $G$ is not available. Thus, we need to find a work around.

By the Brauer-Dade theory already used in Step~3, we have 
\begin{equation}\label{Cu}
C_u=p\biggl(\frac{p-1}{e_u}+\delta_{ij}\biggr)_{i,j=1}^{e_u}
\end{equation}
up to basic sets where $e_u:=\lvert\C_{I(B)}(u)\rvert$.
Since we already know $Q_1$, we can consider the $\mathbb{Z}$-module
$\mathcal{M}:=\{x\in\mathbb{Z}^{k(B)}:Q_1^\text{T}x=0\}$. By the orthogonality relations, the columns of $Q_u$ lie in $\mathcal{M}$. We will often show that $I(B)$ has only one orbit on $D\setminus\{1\}$, i.\,e. there is only one matrix $Q_u$ with $u\in\mathcal{R}$. In this case, let $\widetilde{Q}_u$ be a matrix whose columns form a basis of $\mathcal{M}$. It will turn out that $\det(\widetilde{Q}_u^\text{T}\widetilde{Q}_u)=\det C_u$. Then $Q_u$ is given by $\widetilde{Q}_u$ up to basic sets.

Hence, it remains to consider the case $\lvert\mathcal{R}\rvert\ge 2$. 
This part forms the core of the algorithm. We may assume that $C_u$ is given as above. An algorithm by Plesken~\cite{Plesken} aims to find all integral solutions $X\in\mathbb{Z}^{k(B)\times l(b_u)}$ of the equation $X^\text{T}X=C_u$. Of course, $Q_u$ is one of these solutions. The algorithm starts by computing the possible rows $r$ of $X$. These are (finitely many) integral vectors $0\ne r\in\mathbb{Z}^{l(b_u)}$ such that $m_r:=r^\text{T}C_u^{-1}r\le 1$. In our case, $m_r$ must be a diagonal entry of the contribution matrix $M_u$. Since all irreducible characters of $B$ have height $0$, a result by Brauer implies that $p^2m_r$ is an integer not divisible by $p$ (see \cite[Proposition 1.37]{habil}). Most likely, there are still too many possible rows in order to find all solutions $X$. We will use more theory in order to reduce the number of rows $r$ further. 

Let $\lambda$ be an $I(B)$-stable generalized character of $D$, and let $\chi\in\Irr(B)$. Then by Broué-Puig~\cite{BrouePuigA}, $\chi\mathbin{\ast}\lambda$ is a generalized character of $B$. In particular, the scalar product $(\chi\mathbin{\ast}\lambda,\psi)_G$ for $\psi\in\Irr(B)$ is an integer. This means that $\sum_{u\in\mathcal{R}\cup\{1\}}{\lambda(u)m_{\chi\psi}^u}\in\mathbb{Z}$. Running through all $\chi,\psi\in\Irr(B)$ we get
\[\sum_{u\in\mathcal{R}\cup\{1\}}{\lambda(u)M_u}\in\mathbb{Z}^{k(B)\times k(B)}.\]
Hence, the matrices $Q_u$ are not independent of each other.
Choosing the trivial character $\lambda=1_D$ we obtain the slightly stronger relation 
$\sum_{u\in\mathcal{R}\cup\{1\}}{M_u}=1_{k(B)}$
which was already known to Brauer (see \cite[Lemma~V.9.3]{Feit}). Since we know $M_1$, this gives a better upper bound for $m_r$.
On the other hand, if $\lambda$ is the regular character of $D$, the sum becomes $p^2M_1\in\mathbb{Z}$. This is, however, clear by the definition of $M_1$. In order to construct the $I(B)$-stable generalized characters systematically, we introduce a lemma which works in a more general situation.

\begin{Lemma}\label{stabchar}
Let $B$ be a block of a finite group with defect group $D$ and fusion system $\mathcal{F}$. Then the number of $\mathcal{F}$-conjugacy classes of $D$ coincides with the rank of the $\mathbb{Z}$-module of $\mathcal{F}$-stable generalized characters of $D$.
\end{Lemma}
\begin{proof}
Let $\mathcal{S}$ be a set of representatives for the conjugacy classes of $D$, and let $\Irr(D)=\{\lambda_1,\ldots,\lambda_s\}$. Every $\mathcal{F}$-conjugacy class of $D$ is a union of $D$-conjugacy classes. This gives a partition $\mathcal{S}=\bigcup_{i=1}^t{\mathcal{C}_i}$. 
In case $t=s$ every generalized character of $D$ is $\mathcal{F}$-stable and the result follows. Hence, we may assume $t<s$ in the following.
Choose $u_i\in\mathcal{C}_i$ for $i=1,\ldots,t$. Let $A$ be the matrix consisting of the rows of the form \[(\lambda_j(u_i)-\lambda_j(v):j=1,\ldots,s)\] 
where $v\in\mathcal{C}_i\setminus\{u_i\}$. Then $A\in\mathbb{C}^{(s-t)\times s}$. It is well-known that the values of the irreducible characters of $D$ are cyclotomic integers. Moreover, the Galois group of the corresponding cyclotomic field acts on the columns of $A$. Thus, after choosing a basis for the ring of cyclotomic integers, we may replace $A$ by $\widetilde{A}\in\mathbb{Z}^{(s-t)\times t}$ such that $Ax=0$ if and only if $\widetilde{A}x=0$ for every $x\in\mathbb{Z}^s$. By transforming $\widetilde{A}$ into its Smith normal form, we find a set of $t$ linearly independent vectors $x_1,\ldots,x_t\in\mathbb{Z}^s$ such that $Ax_i=0$ for $i=1,\ldots,t$. For $x_i=(y_1,\ldots,y_s)$ the generalized character $\sum_{j=1}^s{y_j\lambda_j}$ is $\mathcal{F}$-stable. Hence, the rank of the $\mathbb{Z}$-module of $\mathcal{F}$-stable generalized characters of $D$ is at least $t$. On the other hand, the dimension of the space of all $\mathcal{F}$-stable class function on $D$ is obviously $t$. This proves the claim.
\end{proof}

Coming back to our situation where $I(B)$ has at least two orbits on $D\setminus\{1\}$, it follows from Lemma~\ref{stabchar} that there is at least one more “interesting” $I(B)$-stable generalized character on $D$ (apart from the trivial character and the regular character). This greatly reduces the number of possible rows in Plesken's algorithm. 

If Plesken's algorithm yields two different solutions $X_1$ and $X_2$, it may happen that these solutions would coincide up to basic sets (viewed as $Q_u$). It is not obvious to decide when this happens, because the order and signs of the rows of $X_i$ may vary. However, one can proceed as follows. First we make sure that $X_i$ has an integral left inverse. This is often guaranteed, since $X_i$ will contain an identity matrix of size $l(b_u)\times l(b_u)$. Next, we construct $Y_i:=p^2 X_i(X_i^\text{T}X_i)^{-1} X_i^\text{T}$ and check if there is a signed permutation matrix $P$ such that $PY_1P^\text{T}=Y_2$. If this is not the case, then clearly, $X_1$ and $X_2$ are essentially different (i.\,e. not equal up to basic sets). Now assume that $P$ exists. By construction, $X_2$ lies in the eigenspace of $p^2Y_2$ with respect to the eigenvalue $p^2$. If $E_2$ is a basis of the corresponding integral eigenspace, then there exists a square matrix $S_2$ such that $X_2=E_2S_2$. Since $X_2$ has an integral left inverse, we see that also $X_2$ is a basis for the integral eigenspace. Similarly, $PX_1$ is a basis for the integral eigenspace of $Y_2$. Hence, there is a matrix $S\in\GL(l(b_u),\mathbb{Z})$ such that $PX_1S=X_2$, i.\,e. $X_1$ and $X_2$ coincide up to basic sets. The search for the matrix $P$ can be implemented as a backtracking algorithm going through the rows of $Y_i$. 

\subsection{Step~6: Construct perfect isometries}
For the existence of a perfect isometry between $B$ and $b$ we use \cite[Theorem~2]{WatanabePerfIso}. By Step~4, we already know the ordinary decomposition matrix of $b$ up to basic sets. 
Then, by applying the methods from the last paragraph of Step~5, we will show that the ordinary decomposition matrices of $B$ and $b$ coincide up to basic sets and up to permutations of $\Irr(B)$. This is actually necessary for the existence of a perfect isometry. Also note that the existence of an integral left inverse is generally true for ordinary decomposition matrices (see \cite[proof of Theorem~3.6.14]{Nagao}).

Since also the irreducible characters of $b$ are $p$-rational, the generalized decomposition matrices of $b$ 
can be obtained completely analogously to $Q_u$. 
In order to apply \cite[Theorem~2]{WatanabePerfIso}, it suffices to prove that the matrices $Q_u$ are unique in the following sense: Whenever we have two candidates $Q_u$ and $Q_u'$ for each $u\in\mathcal{R}$, there exists a signed permutation matrix $P\in\mathbb{Z}^{k(B)\times k(B)}$ not depending on $u$ and matrices $S_u\in\GL(l(b_u),\mathbb{Z})$ such that $PQ_u'S_u=Q_u$ for every $u\in\mathcal{R}$ (note that we do not need to consider $u=1$). 
We will manage this step in the following manner: 
We fix the order of $\Irr(B)$ such that $M_1$ is uniquely determined by $Q_1$.
Then we use the solutions from Plesken's algorithm in order to show that the matrices $M_u$ are uniquely determined by $M_1$. 
Finally, the eigenspace argument described above allows us to reconstruct $Q_u$ from $M_u$. 
In case $p=3$ we are able to prove this kind of uniqueness for $Q_u$ without using $Q_1$ or $M_1$.

\subsection{Step~7: Extend perfect isometries to isotypies} 
By Step~6, there exists a perfect isometry $I:\CF(G,B)\to\CF(\N_G(D),b)$ where $\CF(G,B)$ denotes the space of class functions with basis $\Irr(B)$ over $K$. It remains to show that $I$ is also an isotypy. In order to do so, we follow \cite[Section~V.2]{Cabanes}. For each $u\in\mathcal{R}\cup\{1\}$ let $\CF(\C_G(u)_{2'},b_u)$ be the space of class functions on $\C_G(u)$ which vanish on the $p$-singular classes and are spanned by $\IBr(b_u)$. The decomposition map $d^u_G:\CF(G,B)\to\CF(\C_G(u)_{p'},b_u)$ is 
defined by 
\[d^u_G(\chi)(s):=\chi(i_{b_u}us)=\sum_{\phi\in\IBr(b_u)}{d^u_{\chi\phi}\phi(s)}\] 
for $\chi\in\Irr(B)$ and $s\in\C_G(u)_{p'}$ where $i_{b_u}$ is the block idempotent of $b_u$ over $\mathcal{O}$. 
Let $\beta_u$ be a Brauer correspondent of $b$ in $\C_G(u)\cap\N_G(D)$. 
Then $I$ determines isometries 
\[I^u:\CF(\C_G(u)_{p'},b_u)\to\CF(\C_{\N_G(D)}(u)_{p'},\beta_u)\] 
by the equation $d^u_{\N_G(D)}\circ I=I^u\circ d^u_G$. Note that $I^1$ is the restriction of $I$. We need to show that $I^u$ can be extended to a perfect isometry $\widehat{I^u}:\CF(\C_G(u),b_u)\to\CF(\C_{\N_G(D)}(u),\beta_u)$ for each $u\in\mathcal{R}$ which does not depend on the generator $u$ of $\langle u\rangle$. By the construction of $I$, we can choose basic sets $\phi_1,\ldots,\phi_l$ (resp. $\widetilde{\phi}_1,\ldots,\widetilde{\phi}_l$) of $b_u$ (resp. $\beta_u$) such that $I^u(\phi_i)=\widetilde{\phi}_i$. The Cartan matrix of $b_u$ and $\beta_u$ with respect to these basic sets is given as in \eqref{Cu}.

The block $b_u$ has defect group $D$, and the focal subgroup $[D,I(b_u)]$ has order at most $p$. These blocks are well understood by results of Watanabe~\cite{Watanabe1,WatanabeCycFoc}. 
In particular, by \cite[Corollary]{WatanabeCycFoc} there exists a perfect isometry between $b_u$ and $\beta_u$, but we must take care of the extension property.
As before let $e_b:=\lvert I(b_u)\rvert=\lvert\C_{I(B)}(u)\rvert$. Then $k(b_u)=p(\frac{p-1}{e_u}+e_u)$ and $l(b_u)=e_u$ (cf. \cite[Lemma~9]{SambaleC4}). We compute the ordinary decomposition matrix of $b_u$. If $e_u=1$, then $b_u$ is nilpotent. In this case there exists a perfect isometry $\widehat{I^u}:\CF(\C_G(u),b_u)\to\CF(\C_{\N_G(D)}(u),\beta_u)$ sending a $p$-rational character $\psi\in\Irr(b_u)$ to a $p$-rational character $\widetilde{\psi}\in\Irr(\beta_u)$. Since $\psi$ is an extension of $\pm\phi_1$, it is clear that $\widehat{I^u}$ extends $I^u$. If we take another generator $v$ of $\langle u\rangle$, then the ordinary decomposition numbers of $\psi$ are of course the same. Therefore, $\widehat{I^u}$ only depends on $\langle u\rangle$. 

It remains to deal with the case where $b_u$ is non-nilpotent. Then $p$ is odd, $|[D,I(b_u)]|=p$ and $D=[D,I(b_u)]\times\C_D(I(b_u))$. Since $[D,I(b_u)]$ acts freely on $\Irr(B)$ via the $*$-construction (see \cite{RobinsonFocal}), the rows of the ordinary decomposition matrix $\Gamma_1$ of $b_u$ come in groups of $p$ equal rows each. Hence, in order to determine $\Gamma_1$ with respect to the chosen basic set $\phi_1,\ldots,\phi_l$, it suffices to solve the matrix equation $X^\text{T}X=\frac{1}{p}C_u$ for $X\in\mathbb{Z}^{\frac{1}{p}k(b_u)\times l(b_u)}$ (i.\,e. we actually compute the decomposition matrix of a block with defect group of order $p$). A detailed analysis of the solutions $X$ has been carried out in \cite[proof of Proposition~6]{SambaleC3}. As a result, one has
\[X=\begin{pmatrix}\label{ind2}
1&&0\\
&\ddots&\\
0&&1\\
1&\cdots&1\\
\vdots&\ddots&\vdots\\
1&\cdots&1
\end{pmatrix},\]
after choosing the order and signs of $\Irr(b_u)$. Then 
\[\Gamma_1=(\underbrace{X^\text{T},\ldots,X^\text{T}}_{p\text{ times}})^\text{T}.\] 
(For an alternative approach, one could also use \cite[Corollary]{WatanabeCycFoc} and compute $\Gamma_1$ in the local situation $D\rtimes I(b_u)$.)
In particular, we may choose $p$-rational (generalized) characters $\psi_1,\ldots,\psi_l\in\pm\Irr(b_u)$ (resp. $\widetilde{\psi}_1,\ldots,\widetilde{\psi}_l\in\pm\Irr(\beta_u)$) such that $\psi_i$ extends $\phi_i$ for $i=1,\ldots,l$. It remains to show that there is a perfect isometry sending $\psi_i$ to $\widetilde{\psi}_i$. 
Let $v\in\C_D(I(b_u))$. Then $d^v_{\psi_i\phi_j}=\zeta d^1_{\psi_i\phi_j}=\zeta\delta_{ij}$ for a $p$-th root of unity $\zeta$. Since $\psi_i$ is $p$-rational and $p>2$, we must have $d^v_{\psi_i\phi_j}=d^1_{\psi_i\phi_j}=\delta_{ij}$. Now let $v\in D\setminus\C_D(I(b_u))$. Then the Brauer correspondent of $b_u$ in $\C_G(u)\cap\C_G(v)$ is nilpotent. Thus, $d^v_{\psi_i\lambda_v}=\pm1$. 
Let us choose signs for $\lambda_v$ such that $d^v_{\psi_1\lambda_v}=1$ for all $v\in D\setminus\C_D(I(b_u))$. We will show now that $d^v_{\psi_i\lambda_v}=1$ for all $v\in D\setminus\C_D(I(b_u))$ and $i=2,\ldots,l$. For this let $\mathcal{S}$ be a set of representatives for the $I(b_u)$-conjugacy classes of $[D,I(b_u)]$. Then $\lvert\mathcal{S}\rvert=1+\frac{p-1}{e_u}$. Moreover, $p^2C_u^{-1}=(-\frac{p-1}{e_u}+p\delta_{ij})$ and $m_{\psi_1\psi_i}^1=-\frac{p-1}{p^2e_u}$. Consequently,
\[0=(\psi_1,\psi_i)_{\C_G(u)}=\sum_{v\in\C_D(I(b_u))}\sum_{w\in\mathcal{S}}{m_{\psi_1\psi_i}^{vw}}=-\frac{p-1}{pe_u}+\frac{1}{p^2}\sum_{ v\in\C_D(I(b_u))}\sum_{1\ne w\in\mathcal{S}}{d^{vw}_{\psi_i\lambda_{vw}}}\]
and $d^{vw}_{\psi_i\lambda_{vw}}=1$. Therefore, for any permutation $\pi\in S_l$ there exists a perfect isometry $\CF(\C_G(u),b_u)\to\CF(\C_G(u),b_u)$ sending $\psi_i$ to $\psi_{\pi(i)}$. Hence, by \cite[Corollary]{WatanabeCycFoc} there exists a perfect isometry $\widehat{I^u}:\CF(\C_G(u),b_u)\to\CF(\C_{\N_G(D)}(u),\beta_u)$ sending $\psi_i$ to $\widetilde{\psi}_i$. In particular $\widehat{I^u}$ extends $I^u$. Since $\widehat{I^u}$ is defined by means of $p$-rational characters, it does not depend on the choice of the generator of $\langle u\rangle$. %Watanabe lemma

Summarizing, we have shown that in our situation the perfect isometry coming from Step~6 can always be extended to an isotypy. Thus, for the proof of Theorem~\ref{main} it suffices to complete the Steps~1 to 6. 

\section{Proof of Theorem~\ref{main}}
\subsection{The case $p=3$}

For $p=3$ we can use the methods described in the previous section to obtain a stronger result which handles all cases immediately.

\begin{Theorem}
Let $B$ and $\widetilde{B}$ be blocks of \textup{(}possibly different\textup{)} finite groups with a common elementary abelian defect group of order $9$ and isomorphic inertial quotients. If $k(B)=k(\widetilde{B})$ and $l(B)=l(\widetilde{B})$, then $B$ and $\widetilde{B}$ are isotypic.
\end{Theorem}
\begin{proof}
We need to discuss various cases according to the parameters $I(B)$, $k(B)$ and $l(B)$. Suppose first that $I(B)$ is cyclic of order at most $4$. Then by \cite{Usami23I,UsamiZ4}, $B$ is isotypic to the principal block of $D\rtimes I(B)$. For $I(B)\cong Z_2$, there are two essentially different actions of $I(B)$ on $D$, but these can be distinguished by the knowledge of $k(B)$ (see \cite{Kiyota}). 
The same arguments apply for $\widetilde{B}$. Hence, $B$ and $\widetilde{B}$ are isotypic. Now let $I(B)\cong Z_2\times Z_2$. Then by \cite{UsamiZ2Z2}, $B$ is isotypic to the principal block of $D\rtimes I(B)$ or to a non-principal block of $D\rtimes Q_8$ where the kernel of the action of $Q_8$ on $D$ has order $2$. Fortunately, these two types can also be distinguished with \cite{Kiyota} (see also \cite{KessarC3C3}). 

Suppose next that $I(B)\in\{Z_8,Q_8\}$. Then $I(B)$ acts regularly on $D\setminus\{1\}$. It follows that $k(B)-l(B)=1$. By \cite{Kiyota}, one has $k(B)\in\{3,6,9\}$ (where the first case would contradict Alperin's Weight Conjecture). In all three cases $Q_u$ ($u\in\mathcal{R}$) is essentially uniquely determined as $(1,2,2)^\text{T}$, $(1,1,1,1,1,2)^\text{T}$ or $(1,1,1,1,1,1,1,1,1)^\text{T}$ respectively. Thus, \cite[Theorem~2]{WatanabePerfIso} implies the claim.

Now let $I(B)\cong SD_{16}$. Here by \cite{WatanabeSD16} we have $k(B)=9$ and $l(B)=7$. Moreover, $I(B)$ acts transitively on $D\setminus\{1\}$, i.\,e. $\lvert\mathcal{R}\rvert=1$. We show first that the characters in $\Irr(B)$ are $p$-rational. Suppose the contrary. We follow the proof of \cite[Proposition~(2E)]{Kiyota}.
Since the decomposition numbers of a block with cyclic defect group are $0$ or $1$, we obtain the exact Cartan matrix (not up to basic sets) of $b_u$ for $u\in\mathcal{R}$ as 
\[C_u=3\begin{pmatrix}
2&1\\1&2
\end{pmatrix}\]
(the exactness is important here). The two columns of $Q_u$ form a pair of complex conjugate vectors. So we only need to determine the first column $q_1$. There are integral vectors $\alpha,\beta\in\mathbb{Z}^9$ such that $q_u=\alpha+\beta\zeta$ where $\zeta=e^{2\pi i/3}$. An easy calculations gives the following scalar products $(\alpha,\alpha)=5$, $(\beta,\beta)=2$ and $(\alpha,\beta)=1$. This means that some entries of $q$ must vanish. However, this contradicts a theorem by Brauer (see \cite[Proposition 1.37]{habil}). Therefore, we have shown that any block with elementary abelian defect group of order $9$ and inertial quotient $SD_{16}$ has only $p$-rational irreducible characters. 
For the determination of $Q_u$ we observe first that all entries of $Q_u$ are $0$ or $\pm1$. We conclude further that

\begin{equation}\label{Qu}
Q_u=\begin{pmatrix}
1&1&1&1&1&1&.&.&.\\
1&1&1&.&.&.&1&1&1
\end{pmatrix}^\text{T}
\end{equation}
up to basic sets and permutation and signs of $\Irr(B)$. Therefore, $Q_u$ is essentially unique and the claim follows.

Finally, it remains to handle the case $I(B)\cong D_8$. Here $\mathcal{R}=\{u,v\}$ where $u$ (resp. $v$) is conjugate to $u^{-1}$ (resp. $v^{-1}$) under $I(B)$.
Moreover, there are two subcases $(k(B),l(B))\in\{(9,5),(6,2)\}$ (see \cite{Kiyota}). Suppose first that $k(B)=9$ and $l(B)=5$. Then the same argument as above shows that all characters in $\Irr(B)$ are $p$-rational. Also, we may assume that $Q_u$ is given as in \eqref{Qu}. Now consider the first row $r$ of $Q_v$. Certainly, we may choose the basic set such that $r\in\{(1,0),(1,1)\}$. In case $r=(1,0)$ we apply the basic set transformation $\bigl(\begin{smallmatrix}
1&1\\0&-1
\end{smallmatrix}\bigr)$. This transformation does not change $C_v$ ($=C_u$) and afterwards $r=(1,1)$. Since $u$ is conjugate to $u^{-1}$ under $I(B)$, one has $D=\langle u,v\rangle$. There is an $I(B)$-stable generalized character $\lambda$ of $D$ given by $\lambda(1)=0$, $\lambda(u)=3$ and $\lambda(v)=6$.
This shows $3M_u+6M_v\in\mathbb{Z}^{9\times 9}$ and thus $9M_u\equiv 9M_v\pmod{3}$. An easy calculation shows that
\[9M_u=\begin{pmatrix}
2J&J&J\\J&2J&-J\\J&-J&2J
\end{pmatrix}\]
where $J$ is the $3\times 3$ matrix whose entries are all $1$. Now consider the first row $m$ of $M_v$. Its first entry is certainly $2$. The second and third entries are $-1$ or $2$. The other entries are $1$ or $-2$. Altogether, $m$ contains three entries $\pm2$ and six entries $\pm1$. Moreover, by the orthogonality relations $M_uM_v=0$. Since we can still permute the characters $4$ to $6$ and $7$ to $9$ without changing $M_u$, we may assume that $m=(2,-1,-1,-2,1,1,-2,1,1)$. By considering other rows of $M_v$ it follows easily that
\[Q_v=\begin{pmatrix}
1&-1&.&-1&1&.&-1&1&.\\
1&.&-1&-1&.&1&-1&.&1
\end{pmatrix}^\text{T}.\]

Now let $k(B)=6$ and $l(B)=2$. It is easy to see that there is no integral matrix $Q_u$ such that $Q_u^\text{T}Q_u=C_u$ with $C_u$ as above. Therefore $B$ contains pairs of $p$-conjugate characters. Since $u$ and $u^{-1}$ are $I(B)$-conjugate, we conclude that the columns of $Q_u$ are complex conjugate to each other. The same must hold for $Q_v$. Therefore, it suffices to determine the first column of $Q_u$ and of $Q_v$. Similarly as above, one can express the first column of $Q_u$ as $q_u=\alpha+\beta\zeta$ where $\zeta=e^{2\pi i/3}$. Then $(\alpha,\alpha)=5$, $(\beta,\beta)=2$ and $(\alpha,\beta)=1$. This gives the following solution
\[\begin{pmatrix}
\alpha\\\beta
\end{pmatrix}=\begin{pmatrix}
1&.&1&1&1&1\\
1&-1&.&.&.&.
\end{pmatrix}.\]
Observe that the first two characters are $p$-conjugate.
For the first column of $Q_v$ we have similar vectors $\alpha'$ and $\beta'$. By Brauer's permutation lemma (see \cite[Lemma~IV.6.10]{Feit}), $B$ has exactly two pairs of $p$-conjugate vectors. Thus, the first two entries of $\beta'$ must vanish. 
The orthogonality relations imply $(\alpha,\alpha')=(\alpha,\beta')=(\beta,\alpha')=(\beta,\beta')=0$.
Since we can still permute the characters $3$ to $6$, we get just two solutions:
\begin{align*}
\begin{pmatrix}
\alpha'\\\beta'
\end{pmatrix}=\begin{pmatrix}
1&1&1&.&-1&-1\\
.&.&1&-1&.&.
\end{pmatrix}&&\text{or}&&\begin{pmatrix}
\alpha'\\\beta'
\end{pmatrix}=\begin{pmatrix}
1&1&-1&.&1&-1\\
.&.&-1&1&.&.
\end{pmatrix}.
\end{align*}
A calculations shows that only in the first case we have $9M_u\equiv 9M_v\pmod{3}$. Hence, $Q_v$ is essentially unique and the proof is complete.
\end{proof}

\subsection{The case $p=5$}\label{secp5}
By Step~1 we have 
\[(k(B),l(B))\in\{(16,12),(20,16),(16,14),(20,14),(14,6)\}.\]
Now we determine $I(B)$. If $k(B)=16$ and $l(B)=12$ it turns out that in all cases the multiplicity of $5$ as an elementary divisor of $C_1$ is $2$. Hence, $\lvert\mathcal{R}\rvert\le 2$. This gives the possibilities $I(B)\cong Z_4\times S_3$ or $I(B)\cong\SL(2,3)\rtimes Z_4\cong\texttt{SmallGroup}(96,67)$. In the second case the Schur multiplier of $I(B)$ is trivial and Alperin's Weight Conjecture (which is known to be true here) would imply $l(B)=k(I(B))=16$. Hence, $I(B)\cong Z_4\times S_3$ and $D\rtimes I(B)\cong \texttt{SmallGroup}(600,151)$ (thus the action is essentially unique).

Now let $k(B)=20$ and $l(B)=16$. Then the multiplicity of $5$ as an elementary divisor of $C_1$ is $3$ in all cases. Thus, $\lvert\mathcal{R}\rvert=1$. This gives $I(B)\cong\SL(2,3)\rtimes Z_4\cong\texttt{SmallGroup}(96,67)$. 
Here $D\rtimes I(B)$ is too large to have a small group id. However, we can describe the action as $D\rtimes I(B)\cong\texttt{PrimitiveGroup}(25,19)$.

Next, let $k(B)=16$ and $l(B)=14$. Here $5$ occurs just once as an elementary divisor of $C_1$. Therefore, $\lvert\mathcal{R}\rvert=1$. It follows that $I(B)\cong\SL(2,3)\rtimes Z_2\cong\texttt{SmallGroup}(48,33)$ or $I(B)\cong Z_{24}\rtimes Z_2\cong\texttt{SmallGroup}(48,5)$.
In the second case the Schur multiplier of $I(B)$ is trivial and Alperin's Weight Conjecture would imply $l(B)=k(I(B))=18$.
Therefore, $I(B)\cong\SL(2,3)\rtimes Z_2$ and $D\rtimes I(B)\cong\texttt{SmallGroup}(1200,947)$. 

Let $k(B)=20$ and $l(B)=14$. Then $5$ occurs with multiplicity $4$ as an elementary divisor of $C_1$. This gives $\lvert\mathcal{R}\rvert\le 2$ and then $I(B)\cong Z_4\wr Z_2$. 
The action is given by $D\rtimes I(B)\cong\texttt{SmallGroup}(800,1191)$.
Finally, let $k(B)=14$ and $l(B)=6$. Here $5$ occurs with multiplicity $2$ as an elementary divisor of $C_1$. It follows that $I(B)\cong D_{12}\cong S_3\times Z_2$ and $D\rtimes I(B)\cong\texttt{SmallGroup}(300,25)$.
We summarize the results in Table~\ref{tab} (we use $BM$ for the Baby Monster, because $B$ is already reserved for the block).

\begin{table}[ht]
\begin{center}
\begin{tabular}{|c|c|c|c|c|c|}
\hline
$G$&no. of block(s)&$k(B)$&$l(B)$&$I(B)$&$D\rtimes I(B)$\\\hline
$2.Suz$&19&\multirow{8}{*}{16}&\multirow{8}{*}{12}&\multirow{8}{*}{$Z_4\times S_3$}&\multirow{8}{*}{$\texttt{SmallGroup}(600,151)$}\\
$Suz.2$&$2^+$&&&&\\
$3.Suz$&18, 19&&&&\\ 
$3.Suz.2$&24*&&&&\\
$6.Suz$&59, 60&&&&\\ 
$6.Suz.2$&53*&&&&\\
$Co_1$&3&&&&\\
$2.J_2.2$&9&&&&\\\hline
$2.Fi_{22}$&39&\multirow{13}{*}{20}&\multirow{13}{*}{16}&\multirow{13}{*}{$\SL(2,3)\rtimes Z_4$}&\multirow{13}{*}{$\texttt{PrimitiveGroup}(25,19)$}\\
$Fi_{22}.2$&$2^+$&&&&\\
$3.Fi_{22}$&39, 40&&&&\\
$3.Fi_{22}.2$&59*&&&&\\
$6.Fi_{22}$&107, 108&&&&\\
$6.Fi_{22}.2$&101*&&&&\\
$Fi_{23}$&2&&&&\\ 
$Fi_{24}'$&3&&&&\\
$Fi_{24}'.2$&$2^+$, 3, $4^+$, $5^+$&&&&\\
$3.Fi_{24}'$&45, 46&&&&\\
$3.Fi_{24}'.2$&69*&&&&\\
$BM$&2, 8&&&&\\
$M$&4&&&&\\\hline
$Fi_{24}'$&2&\multirow{2}{*}{16}&\multirow{2}{*}{14}&\multirow{2}{*}{$\SL(2,3)\rtimes Z_2$}&\multirow{2}{*}{$\texttt{SmallGroup}(1200,947)$}\\
$2.Fi_{22}.2$&59&&&&\\\hline
$3.Fi_{24}'$&47, 48&\multirow{2}{*}{20}&\multirow{2}{*}{14}&\multirow{2}{*}{$Z_4\wr Z_2$}&\multirow{2}{*}{$\texttt{SmallGroup}(800,1191)$}\\
$3.Fi_{24}'.2$&70*&&&&\\\hline
$2.J_2$&6&\multirow{2}{*}{14}&\multirow{2}{*}{6}&\multirow{2}{*}{$D_{12}$}&\multirow{2}{*}{$\texttt{SmallGroup}(300,25)$}\\
$2.Suz.2$&26&&&&\\\hline
\end{tabular}
\end{center}
\caption{faithful non-principal $5$-blocks with non-cyclic abelian defect groups}\label{tab}
\end{table}

If we have two blocks $B$ and $\widetilde{B}$ such that $k(B)=k(\widetilde{B})$ and $l(B)=l(\widetilde{B})$, then by Step~6 above, we can show that the ordinary decomposition matrices of $B$ and $\widetilde{B}$ coincide up to basic sets. Therefore, we only need to handle one example case for each pair $(k(B),l(B))$. The blocks marked with an asterisk in Table~\ref{tab} are Morita equivalent to corresponding blocks of $G'$ via the Fong-Reynolds Theorem. Hence, in this case the decomposition matrices 
are automatically equal. For example, the $24$-th block of $3.Suz.2$ covers the $18$-th and $19$-th block of $3.Suz$. In particular, these two blocks of $3.Suz$ are also Morita equivalent. 
A similar argument applies to the blocks marked with a plus sign. Here, the characters of $\Irr(B)$ restrict to irreducible characters of $G'$, and therefore $B$ is isomorphic to a block of $G'$. For example, the second block of $Suz.2$ is isomorphic to the principal block of $Suz$ (and of $Suz.2$).

In case $G=2.J_2$ there exists a splendid derived equivalence between $B$ and $b$ as it was shown in \cite{HollowayJ2}. This implies the existence of an isotypy. So we do not need to handle this case.

If $I(B)\cong \SL(2,3)\rtimes Z_4$ or $I(B)\cong\SL(2,3)\rtimes Z_2$, then the Schur multiplier of $I(B)$ is trivial. Hence, $b$ is Morita equivalent to the group algebra of $D\rtimes I(B)$. It follows that the characters in $\Irr(b)$ are $p$-rational. Moreover, one has $\mathcal{R}=\{u\}$ and the matrix $Q_u$ can easily be computed as the integral orthogonal space of $Q_1$. Then the claim follows. 

Let $G=Co_1$. Here $k(B)=16$, $l(B)=12$ and 
$I(B)\cong Z_4\times S_3$. Moreover, $I(B)$ has Schur multiplier $Z_2$, so that there are two possible Morita equivalence classes for $b$. In the twisted case we would have $l(b)=6$. Hence, $b$ is Morita equivalent to the group algebra of $D\rtimes I(B)$. In particular, the irreducible characters of $b$ are $p$-rational.

The action of $I(B)$ gives $\mathcal{R}=\{u,v\}$ with $l(b_u)=l(b_v)=2$. From the character table of $G$ we get
\begin{align*}
Q_1&=\begin{pmatrix}
1&.&.&.&.&.&.&.&.&.&.&.\\
.&1&.&.&.&.&.&.&.&.&.&.\\
.&.&1&.&.&.&.&.&.&.&.&.\\
.&.&.&1&.&.&.&.&.&.&.&.\\
.&.&.&.&1&.&.&.&.&.&.&.\\
.&.&.&.&.&1&.&.&.&.&.&.\\
.&.&.&.&.&.&1&.&.&.&.&.\\
.&.&.&.&.&.&.&1&.&.&.&.\\
.&.&.&.&.&.&.&.&1&.&.&.\\
.&.&.&.&.&.&.&.&.&1&.&.\\
.&.&.&.&.&.&.&.&.&.&1&.\\
.&.&.&.&.&.&.&.&.&.&.&1\\
.&.&.&-1&1&.&1&.&.&.&.&.\\
.&.&-1&1&-1&1&.&1&.&.&.&.\\
1&-1&.&-1&1&-1&1&-1&-1&.&1&1\\
.&-1&.&-1&.&-1&.&-1&-1&1&1&1
\end{pmatrix},\\
25M_1&=\begin{pmatrix}
19&2&1&-2&-2&1&-3&1&2&4&-2&-2&-3&1&6&-4\\
2&21&-2&-1&-1&-2&1&-2&-4&2&4&4&1&-2&-2&-2\\
1&-2&19&2&-3&4&3&4&-2&1&2&2&-2&-6&-1&-1\\
-2&-1&2&16&6&-3&4&-3&-1&3&1&1&-6&2&2&-3\\
-2&-1&-3&6&16&2&-6&2&-1&3&1&1&4&-3&2&-3\\
1&-2&4&-3&2&19&-2&-6&-2&1&2&2&3&4&-1&-1\\
-3&1&3&4&-6&-2&16&-2&1&2&-1&-1&6&3&3&-2\\
1&-2&4&-3&2&-6&-2&19&-2&1&2&2&3&4&-1&-1\\
2&-4&-2&-1&-1&-2&1&-2&21&2&4&4&1&-2&-2&-2\\
4&2&1&3&3&1&2&1&2&19&-2&-2&2&1&-4&6\\
-2&4&2&1&1&2&-1&2&4&-2&21&-4&-1&2&2&2\\
-2&4&2&1&1&2&-1&2&4&-2&-4&21&-1&2&2&2\\
-3&1&-2&-6&4&3&6&3&1&2&-1&-1&16&-2&3&-2\\
1&-2&-6&2&-3&4&3&4&-2&1&2&2&-2&19&-1&-1\\
6&-2&-1&2&2&-1&3&-1&-2&-4&2&2&3&-1&19&4\\
-4&-2&-1&-3&-3&-1&-2&-1&-2&6&2&2&-2&-1&4&19
\end{pmatrix}.
\end{align*}
Since $C_u=5\bigl(\begin{smallmatrix}
3&2\\2&3
\end{smallmatrix}\bigr)$, one has $25m^u_{\chi\chi}\ge 2$ and similarly $25m_{\chi\chi}^v\ge 2$. Hence, $25m_{\chi\chi}^u\le 7$. Now Plesken's algorithm yields essentially two possibilities for $Q_u$:
\begin{align*}
&\begin{pmatrix}
2&1&1&1&1&1&.&.&.&.&1&1&1&1&1&1\\
1&2&.&.&.&.&1&1&1&1&1&1&1&1&1&1
\end{pmatrix}^\text{T},\\
&\begin{pmatrix}
2&2&1&.&.&.&.&.&.&.&1&1&1&1&1&1\\
1&1&.&1&1&1&1&1&1&1&1&1&1&1&1&1
\end{pmatrix}^\text{T}.
\end{align*}

We have $25m_{\chi\chi}^u\in\{2,3,7\}$. 
By the $*$-construction, $25M_u\equiv 25M_v\pmod{5}$. 
Suppose that the second possibility above occurs for $Q_u$. Then there are distinct characters $\chi,\psi\in\Irr(B)$ such that $49m_{\chi\psi}^u=\pm7$ and $49m_{\chi\psi}^v=\pm2$. However, this gives the contradiction $m_{\chi\psi}^1+m_{\chi\psi}^u+m_{\chi\psi}^v\ne 0$.
It follows that only the first matrix above can occur for $Q_u$ and for $Q_v$. It remains to verify the uniqueness of these matrices (if $Q_1$ is given as above). We may assume that the first row of $Q_u$ is $(1,0)$ (so that $25m_{1,1}^u=3$). Now $\{25m_{1,3}^u,25m_{1,3}^v\}=\{-3,2\}$. By interchanging $u$ and $v$ if necessary, we may assume that $25m_{1,3}^u=-3$. Then the third row of $Q_u$ is $(-1,0)$. Continuing these arguments gives
\begin{align*}
Q_u=\begin{pmatrix}
1&.\\
-1&-1\\
-1&.\\
-1&-2\\
1&1\\
.&-1\\
2&1\\
.&-1\\
-1&-1\\
.&1\\
1&1\\
1&1\\
-1&-1\\
-1&.\\
-1&.\\
.&-1
\end{pmatrix},&&
Q_v=\begin{pmatrix}
-1&.\\
1&1\\
.&1\\
-1&-1\\
1&2\\
1&.\\
1&1\\
1&.\\
1&1\\
.&-1\\
-1&-1\\
-1&-1\\
-2&-1\\
.&1\\
1&.\\
.&1
\end{pmatrix}.
\end{align*} 
This settles the case $I(B)\cong Z_4\times S_3$.

Now let $B$ be the $47$-th $5$-block of $G=3.Fi_{24}'$. Then $k(B)=20$, $l(B)=14$ and $I(B)\cong Z_4\wr Z_2$. The Schur multiplier of $I(B)$ is $Z_2$. 
In the twisted case one has $l(b)=5$ which is not true. Therefore, $b$ is Morita equivalent to the group algebra of $D\rtimes I(B)$. In particular, the characters in $\Irr(b)$ are $p$-rational. 
We have $\mathcal{R}=\{u,v\}$ such that $l(b_u)=4$ and $l(b_v)=2$. From the character table of $G$ we obtain
\begin{align*}
Q_1&=\begin{pmatrix}
1&.&.&.&.&.&.&.&.&.&.&.&.&.\\
.&1&.&.&.&.&.&.&.&.&.&.&.&.\\
.&.&1&.&.&.&.&.&.&.&.&.&.&.\\
.&.&.&1&.&.&.&.&.&.&.&.&.&.\\
.&.&.&.&1&.&.&.&.&.&.&.&.&.\\
.&.&.&.&.&1&.&.&.&.&.&.&.&.\\
.&.&.&.&.&.&1&.&.&.&.&.&.&.\\
.&.&.&.&.&.&.&1&.&.&.&.&.&.\\
.&.&.&.&.&.&.&.&1&.&.&.&.&.\\
.&.&.&.&.&.&.&.&.&1&.&.&.&.\\
.&.&.&.&.&.&.&.&.&.&1&.&.&.\\
.&.&.&.&.&.&.&.&.&.&.&1&.&.\\
.&.&.&.&.&.&.&.&.&.&.&.&1&.\\
.&.&.&.&.&.&.&.&.&.&.&.&.&1\\
1&-1&-1&1&1&.&.&.&.&.&.&.&.&.\\
.&.&.&.&-1&.&.&1&-1&.&1&.&1&.\\
1&.&-1&1&.&.&-1&.&.&-1&1&1&1&.\\
.&.&.&1&.&-1&-1&.&1&.&.&1&.&.\\
.&1&1&.&-1&-1&-1&1&.&.&1&.&.&1\\
.&1&.&.&-1&-1&-1&.&.&-1&1&1&.&1
\end{pmatrix},\\
25M_1&=\left(\begin{smallmatrix}
18&2&4&-4&-4&-1&2&-1&1&2&-3&1&-1&-2&4&-2&3&-3&3&-1\\
2&18&-4&4&4&1&3&1&-1&3&-2&-1&1&-3&-4&-3&2&-2&2&1\\
4&-4&17&3&3&2&1&-3&-2&-4&1&3&2&-1&-3&-1&-1&1&4&-3\\
-4&4&3&17&-3&3&4&-2&-3&-1&-1&-3&-2&1&3&1&1&4&1&-2\\
-4&4&3&-3&17&-2&-1&3&-3&-1&4&2&3&1&3&-4&1&-1&1&-2\\
-1&1&2&3&-2&17&-4&2&3&1&1&3&-3&4&-3&-1&4&-4&-1&-3\\
2&3&1&4&-1&-4&18&1&4&-2&3&4&1&2&1&2&-3&-2&-3&1\\
-1&1&-3&-2&3&2&1&17&3&-4&-4&3&-3&-1&2&4&-1&1&4&-3\\
1&-1&-2&-3&-3&3&4&3&17&-1&4&-3&3&1&-2&-4&1&4&1&-2\\
2&3&-4&-1&-1&1&-2&-4&-1&18&3&4&1&2&1&2&-3&3&2&-4\\
-3&-2&1&-1&4&1&3&-4&4&3&18&-1&-4&-3&1&2&2&-2&2&1\\
1&-1&3&-3&2&3&4&3&-3&4&-1&17&-2&1&-2&1&1&4&-4&3\\
-1&1&2&-2&3&-3&1&-3&3&1&-4&-2&17&4&-3&4&4&1&-1&-3\\
-2&-3&-1&1&1&4&2&-1&1&2&-3&1&4&18&4&-2&-2&-3&3&4\\
4&-4&-3&3&3&-3&1&2&-2&1&1&-2&-3&4&17&-1&4&1&-1&-3\\
-2&-3&-1&1&-4&-1&2&4&-4&2&2&1&4&-2&-1&18&3&-3&3&-1\\
3&2&-1&1&1&4&-3&-1&1&-3&2&1&4&-2&4&3&18&2&-2&4\\
-3&-2&1&4&-1&-4&-2&1&4&3&-2&4&1&-3&1&-3&2&18&2&1\\
3&2&4&1&1&-1&-3&4&1&2&2&-4&-1&3&-1&3&-2&2&18&4\\
-1&1&-3&-2&-2&-3&1&-3&-2&-4&1&3&-3&4&-3&-1&4&1&4&17
\end{smallmatrix}\right).
\end{align*}
By Plesken's algorithm, there is essentially only one choice for $Q_v$:
\[\begin{pmatrix}
1&1&1&1&1&1&1&1&1&1&1&1&1&1&1&.&.&.&.&.\\
1&1&1&1&1&1&1&1&1&1&.&.&.&.&.&1&1&1&1&1
\end{pmatrix}^\text{T}.\]
There exists an $I(B)$-stable generalized character $\lambda$ of $D$ such that $\lambda(1)=5$, $\lambda(u)=0$ and $\lambda(v)=-5$. Consequently, $5M_1-5M_v\in\mathbb{Z}^{20\times 20}$ and $25M_1\equiv 25M_v\pmod{5}$. Thus, the first row of $25M_v$ has the form $m=(3,?,-1,1,1,-1,?,-1,1,?,?,1,-1,?,-1,?,?,?,?,-1)$ where every question mark stands for $\pm2$ or $\pm3$ (and $\pm2$ occurs five times). By the orthogonality relations, $m$ is orthogonal to all rows of $M_1$. This determines $m$ uniquely as $m=(3,2,-1,1,1,-1,-3,-1,1,-3,2,1,-1,3,-1,3,-2,2,-2,-1)$.
It follows that
\[Q_v=\begin{pmatrix}
1&.&-1&1&1&-1&-1&-1&1&-1&.&1&-1&1&-1&1&.&.&.&-1\\
.&-1&-1&1&1&-1&.&-1&1&.&-1&1&-1&.&-1&.&1&-1&1&-1
\end{pmatrix}^\text{T}.\]
Now we compute $M_u=1_{20}-M_1-M_v$ and this gives 
\[Q_u=
\begin{pmatrix}
1&-1&-1&1&1&1&.&.&.&.&.&.&1&-1&.&.&-1&.&.&1\\
.&.&-1&.&.&1&.&-1&.&-1&.&1&1&-1&1&.&-1&.&1&.\\
.&.&.&1&.&.&-1&.&1&.&.&1&1&-1&1&.&-1&-1&.&1\\
.&.&.&.&1&1&.&-1&1&.&-1&.&.&-1&1&1&-1&.&.&1
\end{pmatrix}^\text{T}.\]

Finally, let $B$ be the $26$-th $5$-block of $G=2.Suz.2$. Then $k(B)=14$, $l(B)=6$ and $I(B)\cong D_{12}$.
Here we have to be more careful, since there are four pairs of $p$-conjugate characters in $\Irr(B)$. Since $D\in\Syl_5(G)$, there are no difficulties in constructing $\N_G(D)$ and its character table. It follows that also $b$ has four pairs of $p$-conjugate characters. 
The Schur multiplier of $I(B)$ is $Z_2$. In the twisted case we would get $l(B)=3$. Hence, $b$ is Morita equivalent to the group algebra of $D\rtimes I(B)$. 
From the character table we obtain
\begin{align*}
Q_1&=\begin{pmatrix}
1&.&.&.&.&.\\
.&1&.&.&.&.\\
.&.&1&.&.&.\\
.&.&.&1&.&.\\
.&.&.&.&1&.\\
.&.&.&.&.&1\\
1&-1&.&.&-1&1\\
1&-1&.&.&-1&1\\
-1&1&1&.&1&.\\
-1&1&1&.&1&.\\
.&1&1&1&1&.\\
.&1&1&1&1&.\\
.&-1&.&-1&-1&1\\
.&-1&.&-1&-1&1
\end{pmatrix},\\
25M_1&=\begin{pmatrix}
13&4&2&-8&4&-2&3&3&-3&-3&2&2&-2&-2\\
4&17&-4&-4&-8&4&-1&-1&1&1&1&1&-1&-1\\
2&-4&13&-2&-4&-8&2&2&3&3&3&3&2&2\\
-8&-4&-2&13&-4&2&2&2&-2&-2&3&3&-3&-3\\
4&-8&-4&-4&17&4&-1&-1&1&1&1&1&-1&-1\\
-2&4&-8&2&4&13&3&3&2&2&2&2&3&3\\
3&-1&2&2&-1&3&8&8&-3&-3&2&2&3&3\\
3&-1&2&2&-1&3&8&8&-3&-3&2&2&3&3\\
-3&1&3&-2&1&2&-3&-3&8&8&3&3&2&2\\
-3&1&3&-2&1&2&-3&-3&8&8&3&3&2&2\\
2&1&3&3&1&2&2&2&3&3&8&8&-3&-3\\
2&1&3&3&1&2&2&2&3&3&8&8&-3&-3\\
-2&-1&2&-3&-1&3&3&3&2&2&-3&-3&8&8\\
-2&-1&2&-3&-1&3&3&3&2&2&-3&-3&8&8
\end{pmatrix}.
\end{align*}
In this order the first six characters are $p$-rational.
One can choose $\mathcal{R}=\{u,u^2,v,v^2\}$ such that $u$ and $u^{-1}$ are conjugate under $I(B)$. Moreover, $l(b_u)=l(b_v)=2$.
It follows that $Q_u$ is algebraically conjugate to $Q_{u^2}$.
Since the $p$-conjugate characters come in pairs, the entries of $Q_u$ must be real. Similarly for $Q_v$. Hence, it suffices to determine $Q_u$ and $Q_v$. Let $\rho:=e^{2\pi i/5}+e^{-2\pi i/5}=\frac{\sqrt{5}-1}{2}$. Then there are matrices $R_1,R_2\in\mathbb{Z}^{14\times 2}$ such that $Q_u=R_1+R_2\rho$. Since $\rho^2=1-\rho$, we obtain
\begin{align*}
C_u=Q_u^\text{T}Q_u&=R_1^\text{T}R_1+R_2^\text{T}R_2+(R_1^\text{T}R_2+R_2^\text{T}R_1-R_2^\text{T}R_2)\rho=R_1^\text{T}R_1+R_2^\text{T}R_2,\\
R_2^\text{T}R_2&=R_1^\text{T}R_2+R_2^\text{T}R_1.
\end{align*}
On the other hand, $Q_{u^2}=R_1+R_2(-1-\rho)=(R_1-R_2)-R_2\rho$. It follows that
\[0=Q_u^\text{T}Q_{u^2}=R_1^\text{T}(R_1-R_2)-R_2^\text{T}R_2.\]
Combining these formulas gives 
\begin{align*}
R_1^\text{T}R_1=3\begin{pmatrix}
3&2\\2&3
\end{pmatrix},&&
R_2^\text{T}R_2=2\begin{pmatrix}
3&2\\2&3
\end{pmatrix},&&
R_1^\text{T}R_2=\begin{pmatrix}
3&2\\2&3
\end{pmatrix}.
\end{align*}
Since $B$ has six $p$-rational characters, the matrix $R_2$ has at most eight non-zero rows. Furthermore, the rows come in pairs of the form $(r,-r)$. Hence, we can apply Plesken's algorithm with the matrix $\bigl(\begin{smallmatrix}
3&2\\2&3
\end{smallmatrix}\bigr)$. This gives the following solution for $R_2$:
\[\begin{pmatrix}
.&.&.&.&.&.&1&-1&.&.&1&-1&1&-1\\
.&.&.&.&.&.&.&.&1&-1&1&-1&1&-1
\end{pmatrix}^\text{T}.\]
If $r$ is the $7$-th row of $R_1$, then $r-(1,0)$ is the $8$-th row of $R_1$ and so on. This gives still several possibilities for $R_1$. By the shape of $M_1$ we see that
\[25m_{\chi\chi}^1=17\Longleftrightarrow 25m_{\chi\chi}^u=25m_{\chi\chi}^{u^2}=25m_{\chi\chi}^v=25m_{\chi\chi}^{v^2}=2.\]
Moreover, there is an $I(B)$-stable generalized character $\lambda$ of $D$ such that $\lambda(1)=15$, $\lambda(u)=\lambda(u^2)=5$ and $\lambda(v)=\lambda(v^2)=0$. This yields $2\cdot 25M_1\equiv M_u+M_{u^2}\pmod{5}$. In particular, if $\chi\in\Irr(B)$ is $p$-rational, we have $25m_{\chi\chi}^1\equiv 25m_{\chi\chi}^u\equiv 25m_{\chi\chi}^{u^2}\pmod{5}$ and similarly for $v$. We conclude that 
\[25m_{\chi\chi}^1=13\Longleftrightarrow 25m_{\chi\chi}^u=25m_{\chi\chi}^{u^2}=25m_{\chi\chi}^v=25m_{\chi\chi}^{v^2}=3.\]
Therefore, $R_1$ has two rows of the form $\pm(1,1)$ corresponding to $p$-rational characters. The other four rows corresponding to $p$-rational rows have the form $\pm(1,0)$ or $\pm(0,1)$. 
By symmetry reasons, the considerations for $Q_u$ also apply for $Q_v$. After interchanging $u$ and $v$ if necessary, we obtain
\begin{align*}
Q_u=\begin{pmatrix}
1&.\\-1&-1\\-1&.\\.&-1\\-1&-1\\.&1\\\multicolumn{2}{c}{?}
\end{pmatrix},&&
Q_v=\begin{pmatrix}
1&.\\-1&-1\\.&-1\\.&-1\\-1&-1\\1&.\\\multicolumn{2}{c}{?}
\end{pmatrix}.
\end{align*}
If $R_1'$ is the matrix consisting of the last eight rows of $R_1$, then $(R_1')^\text{T}R_1'=\bigl(\begin{smallmatrix}
5&4\\4&5
\end{smallmatrix}\bigr)$. Hence, the last eight rows of $Q_u$ are of the form $\pm(1+\rho,1)$, $\pm(-\rho,1)$, $\pm(1,1+\rho)$, $\pm(1,-\rho)$, $\pm(1+\rho,1+\rho)$, $\pm(-\rho,-\rho)$ where the last two possibilities occur twice each. 
The orthogonality $Q_1^\text{T}Q_u=0$ reveals the positions of these rows and their signs. 
Observe here that we are still able to permute each pair of $p$-conjugate characters without disturbing $Q_1$. 
This determines $Q_u$. After interchanging $v$ and $v^2$ if necessary, we also obtain $Q_v$ as follows:
\begin{align*}
Q_u=\begin{pmatrix}
1&.\\-1&-1\\-1&.\\.&-1\\-1&-1\\.&1\\1+\rho&1+\rho\\-\rho&-\rho\\1&1+\rho\\1&-\rho\\-1-\rho&-1-\rho\\\rho&\rho\\-1-\rho&-1\\\rho&-1
\end{pmatrix},&&
Q_v=\begin{pmatrix}
1&.\\-1&-1\\.&-1\\.&-1\\-1&-1\\1&.\\-1&-1-\rho\\-1&\rho\\\rho&\rho\\-1-\rho&-1-\rho\\1+\rho&1\\-\rho&1\\-\rho&-\rho\\1+\rho&1+\rho
\end{pmatrix}.
\end{align*}

\subsection{The case $p=7$}

Here we always have $k(B)=27$ and $l(B)\in\{24,21\}$. Suppose first that $l(B)=24$. Then $7$ occurs with multiplicity $2$ as an elementary divisor of $C_1$. Hence, $\lvert\mathcal{R}\rvert=1$ and $I(B)\cong 2.S_4^-\times Z_3\cong\texttt{SmallGroup}(144,121)$ where $2.S_4^-$ denotes the double cover of $S_4$ which is not isomorphic to $\GL(2,3)$. 
The action is given by $D\rtimes I(B)\cong \texttt{PrimitiveGroup}(49,29)$.
Now assume $l(B)=21$. Then $7$ occurs with multiplicity $4$ as an elementary divisor of $C_1$. Therefore, $\lvert\mathcal{R}\rvert\le 2$ and $I(B)\cong\SL(2,3)\times Z_3$. Here, $D\rtimes I(B)\cong \texttt{PrimitiveGroup}(49,27)$.
The groups appearing on Noeske's list all have trivial outer automorphism group. Thus, we do not extend the list here.

\begin{table}[ht]
\begin{center}
\begin{tabular}{|c|c|c|c|c|c|}
\hline
$G$&no. of block&$k(B)$&$l(B)$&$I(B)$&$D\rtimes I(B)$\\\hline
$BM$&2&\multirow{3}{*}{27}&\multirow{3}{*}{24}&\multirow{3}{*}{$2.S_4^-\times Z_3$}&\multirow{3}{*}{$\texttt{PrimitiveGroup}(49,29)$}\\
$2.BM$&73&&&&\\ 
$M$&2&&&&\\\hline 
$2.Co_1$&46&\multirow{2}{*}{27}&\multirow{2}{*}{21}&\multirow{2}{*}{$\SL(2,3)\times Z_3$}&\multirow{2}{*}{$\texttt{PrimitiveGroup}(49,27)$}\\
$BM$&4&&&&\\\hline
\end{tabular}
\end{center}
\caption{faithful non-principal $7$-blocks with non-cyclic abelian defect groups}
\end{table}

In case $I(B)\cong 2.S_4^-\times Z_3$ we have $\lvert\mathcal{R}\rvert=1$ and the Schur multiplier of $I(B)$ is trivial. Hence, there is not much to do in order to establish the isotypy. For the other inertial quotient we will only handle one example case.

Let $B$ be the fourth $7$-block of $G=BM$. Here, $k(B)=27$, $l(B)=21$ and 
$I(B)\cong\SL(2,3)\times Z_3$. The Schur multiplier of $I(B)$ is $Z_3$. If the $2$-cocycle of $b$ is non-trivial, one gets $l(B)=5$ which is not the case. Therefore, $b$ is Morita equivalent to the group algebra of $D\rtimes I(B)$, and the irreducible characters of $b$ are all $p$-rational. 
The action of $I(B)$ gives $\mathcal{R}=\{u,v\}$ with $l(b_u)=l(b_v)=3$. Moreover,
\begin{align*}
Q_1&=\begin{pmatrix}
1&.&.&.&.&.&.&.&.&.&.&.&.&.&.&.&.&.&.&.&.\\
.&1&.&.&.&.&.&.&.&.&.&.&.&.&.&.&.&.&.&.&.\\
.&.&1&.&.&.&.&.&.&.&.&.&.&.&.&.&.&.&.&.&.\\
.&.&.&1&.&.&.&.&.&.&.&.&.&.&.&.&.&.&.&.&.\\
.&.&.&.&1&.&.&.&.&.&.&.&.&.&.&.&.&.&.&.&.\\
.&.&.&.&.&1&.&.&.&.&.&.&.&.&.&.&.&.&.&.&.\\
.&.&.&.&.&.&1&.&.&.&.&.&.&.&.&.&.&.&.&.&.\\
.&.&.&.&.&.&.&1&.&.&.&.&.&.&.&.&.&.&.&.&.\\
.&.&.&.&.&.&.&.&1&.&.&.&.&.&.&.&.&.&.&.&.\\
.&.&.&.&.&.&.&.&.&1&.&.&.&.&.&.&.&.&.&.&.\\
.&.&.&.&.&.&.&.&.&.&1&.&.&.&.&.&.&.&.&.&.\\
.&.&.&.&.&.&.&.&.&.&.&1&.&.&.&.&.&.&.&.&.\\
.&.&.&.&.&.&.&.&.&.&.&.&1&.&.&.&.&.&.&.&.\\
.&.&.&.&.&.&.&.&.&.&.&.&.&1&.&.&.&.&.&.&.\\
.&.&.&.&.&.&.&.&.&.&.&.&.&.&1&.&.&.&.&.&.\\
.&.&.&.&.&.&.&.&.&.&.&.&.&.&.&1&.&.&.&.&.\\
.&.&.&.&.&.&.&.&.&.&.&.&.&.&.&.&1&.&.&.&.\\
.&.&.&.&.&.&.&.&.&.&.&.&.&.&.&.&.&1&.&.&.\\
.&.&.&.&.&.&.&.&.&.&.&.&.&.&.&.&.&.&1&.&.\\
.&.&.&.&.&.&.&.&.&.&.&.&.&.&.&.&.&.&.&1&.\\
.&.&.&.&.&.&.&.&.&.&.&.&.&.&.&.&.&.&.&.&1\\
1&-1&1&-2&.&1&-1&1&.&.&-1&-1&1&-1&.&-1&1&1&1&.&.\\
1&-1&.&-1&.&.&.&1&.&-1&.&-1&1&.&.&.&.&1&.&1&.\\
.&.&.&.&-1&.&.&.&.&.&1&1&-1&.&1&.&.&.&.&.&.\\
1&.&.&.&1&.&1&1&.&-1&.&-1&1&.&-1&.&-1&.&.&1&1\\
1&-1&1&-1&.&1&-1&1&1&.&.&-1&.&-1&.&.&.&.&1&.&1\\
2&-1&1&-2&.&2&-1&1&1&.&-1&-2&1&-1&-1&-1&.&1&1&1&1
\end{pmatrix},\\
49M_1&=\left(
\begin{smallmatrix}
39&3&-2&6&2&-6&-1&-6&-3&4&-2&5&-2&2&2&3&4&-4&-2&-8&-6&-1&1&5&3&-1&4\\
3&39&2&-6&-2&-1&-6&6&3&-4&2&-5&2&-2&5&4&3&4&2&1&-1&1&-8&2&4&-6&3\\
-2&2&43&4&-1&-4&4&-4&-2&-2&1&1&1&6&-1&2&-2&2&-6&4&-4&4&-4&1&2&4&-2\\
6&-6&4&37&-4&5&-5&5&-1&-1&-3&-3&4&-4&3&-6&6&8&4&2&-2&-5&-2&-3&1&2&-1\\
2&-2&-1&-4&36&4&-4&-3&2&2&6&6&-8&1&8&-2&2&5&-1&3&-3&3&-3&-8&5&3&-5\\
-6&-1&-4&5&4&37&5&2&-6&-6&3&3&3&4&4&6&1&-1&-4&-2&-5&-2&-5&3&-1&-2&8\\
-1&-6&4&-5&-4&5&37&-2&6&6&-3&-3&-3&-4&3&1&6&1&4&-5&-2&2&-2&4&8&-5&-1\\
-6&6&-4&5&-3&2&-2&37&1&8&-4&3&-4&4&-3&-1&1&-1&-4&-2&-5&5&2&3&6&5&-6\\
-3&3&-2&-1&2&-6&6&1&39&-3&-2&5&5&2&2&-4&4&3&-2&-1&-6&-8&1&-2&-4&6&4\\
4&-4&-2&-1&2&-6&6&8&-3&39&5&-2&5&2&2&3&-3&3&-2&6&1&-1&-6&-2&-4&-1&4\\
-2&2&1&-3&6&3&-3&-4&-2&5&36&-6&8&-1&-8&-5&5&2&1&-3&-4&-3&3&8&2&4&-2\\
5&-5&1&-3&6&3&-3&3&5&-2&-6&36&8&-1&-8&2&-2&2&1&4&3&4&-4&8&2&-3&-2\\
-2&2&1&4&-8&3&-3&-4&5&5&8&8&36&-1&6&2&-2&-5&1&-3&3&4&3&-6&2&-3&-2\\
2&-2&6&-4&1&4&-4&4&2&2&-1&-1&-1&43&1&-2&2&-2&6&-4&4&-4&4&-1&-2&-4&2\\
2&5&-1&3&8&4&3&-3&2&2&-8&-8&6&1&36&-2&-5&-2&-1&3&4&3&4&6&-2&3&-5\\
3&4&2&-6&-2&6&1&-1&-4&3&-5&2&2&-2&-2&39&3&4&2&1&-1&-6&6&-5&-3&8&-4\\
4&3&-2&6&2&1&6&1&4&-3&5&-2&-2&2&-5&3&39&-4&-2&6&8&6&1&-2&-4&-1&-3\\
-4&4&2&8&5&-1&1&-1&3&3&2&2&-5&-2&-2&4&-4&39&2&-6&6&1&6&2&-3&-6&3\\
-2&2&-6&4&-1&-4&4&-4&-2&-2&1&1&1&6&-1&2&-2&2&43&4&-4&4&-4&1&2&4&-2\\
-8&1&4&2&3&-2&-5&-2&-1&6&-3&4&-3&-4&3&1&6&-6&4&37&-2&-5&5&4&1&-5&6\\
-6&-1&-4&-2&-3&-5&-2&-5&-6&1&-4&3&3&4&4&-1&8&6&-4&-2&37&-2&-5&3&6&5&1\\
-1&1&4&-5&3&-2&2&5&-8&-1&-3&4&4&-4&3&-6&6&1&4&-5&-2&37&5&-3&-6&2&6\\
1&-8&-4&-2&-3&-5&-2&2&1&-6&3&-4&3&4&4&6&1&6&-4&5&-5&5&37&3&6&-2&1\\
5&2&1&-3&-8&3&4&3&-2&-2&8&8&-6&-1&6&-5&-2&2&1&4&3&-3&3&36&-5&4&-2\\
3&4&2&1&5&-1&8&6&-4&-4&2&2&2&-2&-2&-3&-4&-3&2&1&6&-6&6&-5&39&1&3\\
-1&-6&4&2&3&-2&-5&5&6&-1&4&-3&-3&-4&3&8&-1&-6&4&-5&5&2&-2&4&1&37&6\\
4&3&-2&-1&-5&8&-1&-6&4&4&-2&-2&-2&2&-5&-4&-3&3&-2&6&1&6&1&-2&3&6&39
\end{smallmatrix}\right).
\end{align*}
It follows that $49m_{\chi\chi}^1\ge 36$. Hence, $49m_{\chi\chi}^u\le 11$. Plesken's algorithm finds six possibilities for $Q_u$:
\begin{align*}
\begin{pmatrix}
2&1&1\\
1&2&1\\
1&1&2\\
1&1&.\\
1&1&.\\
1&1&.\\
1&.&1\\
1&.&1\\
1&.&1\\
.&1&1\\
.&1&1\\
.&1&1\\
1&.&.\\
1&.&.\\
1&.&.\\
.&1&.\\
.&1&.\\
.&1&.\\
.&.&1\\
.&.&1\\
.&.&1\\
1&1&1\\
1&1&1\\
1&1&1\\
1&1&1\\
1&1&1\\
1&1&1
\end{pmatrix},&&
\begin{pmatrix}
2&1&1\\
2&1&1\\
1&2&1\\
1&1&.\\
1&1&.\\
1&.&1\\
1&.&1\\
1&.&1\\
.&1&1\\
.&1&1\\
.&1&1\\
.&1&1\\
1&.&.\\
.&1&.\\
.&1&.\\
.&1&.\\
.&.&1\\
.&.&1\\
.&.&1\\
.&.&1\\
.&.&1\\
1&1&1\\
1&1&1\\
1&1&1\\
1&1&1\\
1&1&1\\
1&1&1
\end{pmatrix},&&
\begin{pmatrix}
2&1&1\\
2&1&1\\
2&1&1\\
1&2&1\\
1&.&1\\
.&1&1\\
.&1&1\\
.&1&1\\
1&.&.\\
.&1&.\\
.&1&.\\
.&1&.\\
.&1&.\\
.&1&.\\
.&.&1\\
.&.&1\\
.&.&1\\
.&.&1\\
.&.&1\\
.&.&1\\
.&.&1\\
1&1&1\\
1&1&1\\
1&1&1\\
1&1&1\\
1&1&1\\
1&1&1
\end{pmatrix},&&
\begin{pmatrix}
2&1&1\\
2&1&1\\
1&2&1\\
1&2&1\\
1&.&1\\
1&.&1\\
.&1&1\\
.&1&1\\
1&.&.\\
1&.&.\\
1&.&.\\
.&1&.\\
.&1&.\\
.&1&.\\
.&.&1\\
.&.&1\\
.&.&1\\
.&.&1\\
.&.&1\\
.&.&1\\
.&.&1\\
1&1&1\\
1&1&1\\
1&1&1\\
1&1&1\\
1&1&1\\
1&1&1
\end{pmatrix},&&
\begin{pmatrix}
2&1&1\\
2&1&1\\
1&2&1\\
1&1&2\\
1&1&.\\
1&.&1\\
.&1&1\\
.&1&1\\
1&.&.\\
1&.&.\\
1&.&.\\
.&1&.\\
.&1&.\\
.&1&.\\
.&1&.\\
.&1&.\\
.&.&1\\
.&.&1\\
.&.&1\\
.&.&1\\
.&.&1\\
1&1&1\\
1&1&1\\
1&1&1\\
1&1&1\\
1&1&1\\
1&1&1
\end{pmatrix},&&
\begin{pmatrix}
2&1&1\\
1&2&1\\
1&1&.\\
1&1&.\\
1&1&.\\
1&1&.\\
1&.&1\\
1&.&1\\
1&.&1\\
1&.&1\\
1&.&1\\
.&1&1\\
.&1&1\\
.&1&1\\
.&1&1\\
.&1&1\\
1&.&.\\
.&1&.\\
.&.&1\\
.&.&1\\
.&.&1\\
1&1&1\\
1&1&1\\
1&1&1\\
1&1&1\\
1&1&1\\
1&1&1
\end{pmatrix}.
\end{align*}
We will show that only the first matrix occurs for $Q_u$ and $Q_v$. By the $*$-construction, we have $49M_u\equiv 49M_v\pmod{7}$ and similarly $3\cdot 49M_1\equiv 49M_u\pmod{7}$. This shows that only the two first possibilities can occur. Moreover,
\begin{align*}
49m_{\chi\chi}^1&=37\Longleftrightarrow 49m_{\chi\chi}^u=6\Longleftrightarrow 49m_{\chi\chi}^v=6,\\
49m_{\chi\chi}^1&=39\Longleftrightarrow 49m_{\chi\chi}^u=5\Longleftrightarrow 49m_{\chi\chi}^v=5,\\
49m_{\chi\chi}^1&=43\Longrightarrow 49m_{\chi\chi}^u=3,\\
49m_{\chi\chi}^1&=36\Longleftrightarrow\{49m_{\chi\chi}^u,49m_{\chi\chi}^v\}=\{3,10\}.
\end{align*}
We consider the submatrix of $49M_1$ corresponding to the characters $\chi\in\Irr(B)$ such that $49m_{\chi\chi}^1=36$:
\[\begin{pmatrix}
36&6&6&-8&8&-8\\
6&36&-6&8&-8&8\\
6&-6&36&8&-8&8\\
-8&8&8&36&6&-6\\
8&-8&-8&6&36&6\\
-8&8&8&-6&6&36
\end{pmatrix}\]
By way of contradiction, suppose that the second choice above occurs for $M_u$. Then there are distinct characters $\chi,\psi\in\Irr(B)$ such that $49m_{\chi\chi}^1=49m_{\chi\chi}^1=36$, $49m_{\chi\psi}^u=\pm10$ and $49m_{\chi\psi}^v=\pm3$. Hence, $49m_{\chi\psi}^u+49m_{\chi\psi}^v=\pm13$. This gives the contradiction $m_{\chi\psi}^1+m_{\chi\psi}^u+m_{\chi\psi}^v\ne 0$. Therefore, the first matrix above occurs for $Q_u$ and $Q_v$. 

We still need to work in order to show the uniqueness if $Q_1$ is given as above. By permuting the elements of a basic set for $b_u$, we may assume that the first row of $Q_u$ is $(1,0,0)$ (here $49m_{1,1}^1=39$). Since $49m_{2,2}^1=39$ and $49m_{1,2}^1=3$, we have $\{49m_{1,2}^u,49m_{1,2}^v\}=\{-5,2\}$. Thus, after interchanging $u$ and $v$ is necessary, we may assume that the second row of $Q_u$ is $(-1,0,0)$. Similarly, the tenth row of $Q_u$ can be chosen to be $(0,1,0)$. This already determines all the remaining entries of $Q_u$. We get:
\begin{align*}
Q_u=\begin{pmatrix}
1&.&.\\
-1&.&.\\
1&1&1\\
-1&-1&.\\
1&1&2\\
1&1&.\\
-1&-1&.\\
1&.&1\\
1&.&.\\
.&1&.\\
-1&-2&-1\\
-2&-1&-1\\
1&1&1\\
-1&-1&-1\\
-1&-1&-1\\
.&-1&.\\
.&1&.\\
.&.&-1\\
1&1&1\\
.&-1&-1\\
1&.&1\\
.&-1&-1\\
.&1&1\\
1&1&1\\
.&.&-1\\
-1&.&-1\\
.&.&1
\end{pmatrix},&&
Q_v=\begin{pmatrix}
1&.&.\\
.&-1&.\\
1&1&1\\
-1&-1&.\\
-1&-1&-1\\
1&.&1\\
.&-1&-1\\
1&1&.\\
.&.&1\\
.&.&1\\
1&1&1\\
1&1&1\\
-1&-1&-2\\
-1&-1&-1\\
1&2&1\\
-1&.&.\\
.&1&.\\
.&.&-1\\
1&1&1\\
.&-1&-1\\
1&.&1\\
-1&-1&.\\
1&.&1\\
-2&-1&-1\\
-1&.&.\\
.&-1&-1\\
.&1&.
\end{pmatrix}.
\end{align*}
This completes the proof of Theorem~\ref{main}.

\section{Concluding remarks}
The implementation of Plesken's algorithm in GAP is given by the function \texttt{OrthogonalEmbeddings}. There has been a serious bug in this function which was eventually fixed in version 4.7.6 by Thomas Breuer.

Our method does not only establish isotypies between $B$ and its Brauer correspondent $b$, but also between different blocks of sporadic groups with the same inertial quotient $I(B)$.

For some of the smaller sporadic groups one can avoid Plesken's algorithm with the following alternative approach: For $u\in\mathcal{R}$ determine the possible class fusions from $\C_G(u)$ to $G$. 
Then compute the ordinary decomposition matrix up to basic sets for every Brauer correspondent of $B$ in $\C_G(u)$ (see Step~2). By restricting the irreducible characters of $B$ to $\C_G(u)$, one obtains $Q_u$ up to basic sets.
In turns out that $Q_u$ is usually not affected by the choice of the class fusion. However, this procedure has to be done for every group individually. 

\section*{Acknowledgment}
This work is supported by the German Research Foundation and the Daimler and Benz Foundation. The author thanks Jianbei An, Jürgen Müller and Raphaël Rouquier for answering some questions.

\end{document}